\documentclass[11pt, reqno]{amsart}
\usepackage{amsmath,amssymb,amsthm,amscd,latexsym}
\usepackage[dvipdfmx]{graphicx}
\usepackage[all,cmtip]{xy}
\usepackage{wrapfig}
\usepackage{color}
\usepackage{enumerate}
\usepackage{cases}
\usepackage[top=30mm,bottom=30mm,left=30mm,right=30mm]{geometry}
\usepackage{graphics}
\usepackage{tikz}
\usepackage{pxpgfmark}
\usepackage{multirow}
\usepackage{extarrows}
\usepackage{comment}
\usepackage[colorlinks=true, linkcolor=blue, citecolor=red]{hyperref}
\usepackage{bookmark}

\usetikzlibrary{patterns}
\usetikzlibrary{intersections, calc,decorations.markings}
    
\numberwithin{equation}{section}
\theoremstyle{definition}
\newtheorem{example}{Example}[section]
\newtheorem{definition}[example]{Definition}

\newtheorem{remark}[example]{Remark}

\theoremstyle{plain}

\newtheorem{theorem}[example]{Theorem}
\newtheorem{proposition}[example]{Proposition}
\newtheorem{corollary}[example]{Corollary}

\newlength\squareheight
\newcommand{\relmiddle}[1]{\mathrel{}\middle#1\mathrel{}}

\DeclareMathOperator{\diag}{{\rm diag}}

\title[Finite type and completeness of $g$-fans]
{Finite type and completeness of $g$-fans}
\author{Toshiya Yurikusa}
\address{Department of Mathematics, Osaka Metropolitan University, Osaka 558-8585, Japan}
\email{yurikusa@omu.ac.jp}

\begin{document}

\begin{abstract}
We study the $g$-fan associated with a skew-symmetrizable matrix in the sense of cluster algebras. We show that a skew-symmetrizable matrix is of finite type if and only if its $g$-fan is complete; equivalently (as we show), its support contains all lattice points.
\end{abstract}

\maketitle
\tableofcontents

\section{Introduction}

Cluster algebras \cite{FZ02} have been studied from many points of view. One of the basic objects in this theory is the notion of $g$-vectors \cite{FZ07}.

Let $B$ be a skew-symmetrizable matrix. The $g$-vectors associated with $B$ span simplicial cones in $\mathbb{R}^n$. These cones form a simplicial fan, which is called the $g$-fan of $B$ (see \cite{GHKK18,Re14}). The geometry of this fan encodes the structure of the associated cluster algebra.

A skew-symmetrizable matrix $B$ is said to be of \emph{finite type} if it has only finitely many $g$-vectors, or equivalently, finitely many cluster variables. Finite type matrices are well understood and play an important role in cluster theory.

It is then natural to ask how the finite type property of $B$ is reflected in the geometry of its $g$-fan. In particular, one may ask whether the $g$-fan covers the whole space $\mathbb{R}^n$, or whether it can be characterized by the property of containing all lattice points in $\mathbb{Z}^n$. This question is also motivated by \cite[Question~3.49]{D17} in representation theory. In this paper, we give a complete answer to this question in the framework of $g$-fans arising from skew-symmetrizable matrices. More precisely, we prove the following theorem.

\begin{theorem}\label{thm:main}
For a skew-symmetrizable matrix $B$, the following are equivalent:
\begin{itemize}
\item[$(1)$] $B$ is of finite type.
\item[$(2)$] The $g$-fan $\mathcal{F}(B)$ is complete, that is, its support $|\mathcal{F}(B)|=\mathbb{R}^n$.
\item[$(3)$] The support $|\mathcal{F}(B)|$ contains all lattice points in $\mathbb{Z}^n$.
\end{itemize}
\end{theorem}

The implication $(1)\Rightarrow(2)$ is known (Theorem~\ref{thm:Reading}). For skew-symmetric matrices, the implication $(2)\Rightarrow(1)$ was proved in \cite{HY25} via categorification using Jacobian algebras. In \cite{Y23}, we suggested that this implication might be proved by methods similar to those in \cite[Proposition~4.9]{A21} in general, but no complete proof was available. In this paper, we give a complete proof using scattering diagrams and their pull-back constructions.

\section{Mutations of matrices and $g$-vector tuples}\label{sec:}

In this section, we briefly recall mutations of matrices and $g$-vectors \cite{FZ07}. Let $n$ be a positive integer. An $n\times n$ integer matrix $B=(b_{ij})$ is said to be \emph{skew-symmetrizable} if there is a diagonal matrix $D=\diag(d_1,\ldots,d_n)$ with positive integers $d_i$ such that $DB$ is skew-symmetric, that is, $d_i b_{ij}=-d_j b_{ji}$. We recall the mutation of matrices.

\begin{definition}\label{def:matrix mutation}
Let $m\ge n$ be positive integers and $B=(b_{ij})$ an $m\times n$ integer matrix whose upper part $(b_{ij})_{1\le i,j\le n}$ is skew-symmetrizable. The \emph{mutation} of $B$ at $k$ $(1\le k\le n)$ is the matrix $\mu_k(B)=(b'_{ij})$ defined by
\[
b'_{ij}=\begin{cases}
-b_{ij}&\text{if $i=k$ or $j=k$},\\
 b_{ij}+b_{ik}[b_{kj}]_++[-b_{ik}]_+\, b_{kj}&\text{otherwise},
\end{cases}
\]
where $[a]_+:=\max(a,0)$.
\end{definition}

It is straightforward to check that the upper part of $\mu_k(B)$ is again skew-symmetrizable, and that $\mu_k$ is an involution.

An $n\times n$ skew-symmetrizable matrix $B$ is called
\begin{itemize}
\item \emph{mutation equivalent} to a matrix $B'$ if it is obtained from $B'$ by a finite sequence of mutations;
\item \emph{$2$-finite} if for any matrix $B'=(b'_{ij})$ mutation equivalent to $B$, $|b'_{ij}b'_{ji}|\le 3$ for all $i,j$.
\end{itemize}

Although the notion of a cluster algebra provides the natural background for our study, the results of this paper depend only on the mutation rule for $g$-vectors associated with a skew-symmetrizable matrix. For this reason, we do not recall the general definition of cluster algebras.

\begin{definition}
A \emph{$g$-vector seed} is a pair $(C,(\mathbf{g}_1,\ldots,\mathbf{g}_n))$ consisting of the following data:
\begin{itemize}
\item[$(1)$] $C$ is a $2n\times n$ integer matrix whose upper part is skew-symmetrizable.
\item[$(2)$] $\mathbf{g}_1,\ldots,\mathbf{g}_n\in\mathbb Z^n$.
\end{itemize}
The tuple $(\mathbf{g}_1,\ldots,\mathbf{g}_n)$ is called the \emph{$g$-vector tuple}.
\end{definition}

Fix an $n\times n$ skew-symmetrizable matrix $B=(b_{ij})$. Let $(C=(c_{ij}),(\mathbf{g}_1,\ldots,\mathbf{g}_n))$ be a $g$-vector seed. The \emph{mutation} of this seed at $k\in\{1,\ldots,n\}$ is defined by
\[
\mu_k(C,(\mathbf{g}_1,\ldots,\mathbf{g}_n)):=\bigl(\mu_k(C),(\mathbf{g}'_1,\ldots,\mathbf{g}'_n)\bigr),
\]
where
\[
\mathbf{g}'_\ell=\begin{cases}
\mathbf{g}_\ell&\text{if $\ell\neq k$},\\
-\mathbf{g}_k+\displaystyle\sum_{i=1}^n[c_{ik}]_+\,\mathbf{g}_i-\displaystyle\sum_{j=1}^n[c_{n+j,k}]_+\,\mathbf{b}_j&\text{if $\ell=k$},
\end{cases}
\]
and $\mathbf{b}_j$ denotes the $j$th column of $B$.

We denote by $\hat B$ the $2n\times n$ matrix whose upper part is $B$ and whose lower part is the $n\times n$ identity matrix. A $g$-vector tuple is called a \emph{$g$-vector tuple for $B$} if it is obtained from the initial seed $(\hat{B},(\mathbf{e}_1,\ldots,\mathbf{e}_n))$ by a finite sequence of mutations, where $\mathbf{e}_i$ is the $i$th standard basis vector of $\mathbb{Z}^n$. This formulation is equivalent to the usual definition of $g$-vectors in cluster algebras \cite{FZ07}.

\begin{definition}
A skew-symmetrizable matrix $B$ is said to be \emph{of finite type} if there are only finitely many $g$-vector tuples for $B$.
\end{definition}

\begin{remark}
A skew-symmetrizable matrix $B$ is of finite type in the above sense if and only if the associated cluster algebra $\mathcal{A}(B)$ has only finitely many cluster variables.
\end{remark}

\begin{theorem}[{\cite[Theorems~1.5 and~7.1]{FZ03}}]\label{thm:fin 2-fin}
A skew-symmetrizable matrix $B$ is of finite type if and only if it is $2$-finite.
\end{theorem}

Let $g=(\mathbf{g}_1,\ldots,\mathbf{g}_n)$ be a $g$-vector tuple for $B$. We define the \emph{$g$-cone} of $g$ by
\[
C(g):=\left\{\sum_{i=1}^{n}a_i\mathbf{g}_i\relmiddle| a_i\in\mathbb{R}_{\ge 0}\right\}.
\]
We denote by $\mathcal{F}(B)$ the set of all faces of the cones $C(g)$, where $g$ runs over all $g$-vector tuples for $B$. The set $\mathcal{F}(B)$ forms a simplicial polyhedral fan in $\mathbb{R}^n$ \cite[Theorem~0.8]{GHKK18}.

\begin{definition}
The fan $\mathcal{F}(B)$ is called the \emph{$g$-fan} of $B$.
\end{definition}

We denote by $|\mathcal{F}(B)|$ the support of the $g$-fan $\mathcal{F}(B)$, that is
\[
|\mathcal{F}(B)|:=\bigcup_{C \in \mathcal{F}(B)}C\subseteq\mathbb{R}^n.
\]

\begin{example}\label{exam:rank 2 g-fan}
Any nonzero $2 \times 2$ skew-symmetrizable matrix is given by
\[
B_{b,c}:=\begin{bmatrix}0 & c \\ -b & 0\end{bmatrix},
\]
where $b$ and $c$ are integers and $bc > 0$. The $g$-fan $\mathcal{F}(B_{b,c})$ is well-known (see e.g. \cite[Example~1.15]{GHKK18} or \cite[Section~2]{N24}). In fact, for $bc \ge 4$, $\mathcal{F}(B_{b,c})$ contains infinitely many rays converging to the rays $r_{\pm}$ of slope $(-bc \pm \sqrt{bc(bc-4)})/2c$ (see Figure~\ref{fig:rank 2 g-vector fans}). In this case, the ray of slope $-b/2$ is contained in $(\mathbb{R}^2\setminus|\mathcal{F}(B_{b,c})|)\cup\{0\}$. In particular, the lattice point $(-2,b)$ is contained in $\mathbb{Z}^n\setminus|\mathcal{F}(B_{b,c})|$.
\end{example}

\begin{figure}[htp]
\centering
\begin{tabular}{ccc}
\begin{tikzpicture}[baseline=-1mm,scale=1.5]
\coordinate(0)at(0,0); \coordinate(x)at(1,0); \coordinate(-x)at(-1,0); \coordinate(y)at(0,1); \coordinate(-y)at(0,-1);
\draw[->](-x)--(x); \draw[->](-y)--(y); \draw(-0.5,1)--(0)--(-1,1);
\end{tikzpicture}
&
\begin{tikzpicture}[baseline=-1mm,scale=1.5]
\coordinate(0)at(0,0); \coordinate(x)at(1,0); \coordinate(-x)at(-1,0); \coordinate(y)at(0,1); \coordinate(-y)at(0,-1);
\draw[->](-x)--(x); \draw[->](-y)--(y); \draw(-0.25,1)--(0)--(-1,1); \draw(-0.33,1)--(0)--(-0.75,1); \draw(-0.375,1)--(0)--(-0.66,1);
 \draw(-0.4,1)--(0)--(-0.6,1); \draw(-0.42,1)--(0)--(-0.57,1); \draw(-0.44,1)--(0)--(-0.55,1); \draw(-0.46,1)--(0)--(-0.53,1);
\draw[red](0)--(-0.5,1); \node at(-0.5,1.15){$r_+=r_-$};
\end{tikzpicture}
&
\begin{tikzpicture}[baseline=-1mm,scale=1.5]
\coordinate(0)at(0,0); \coordinate(x)at(1,0); \coordinate(-x)at(-1,0); \coordinate(y)at(0,1); \coordinate(-y)at(0,-1);
\draw[->](-x)--(x); \draw[->](-y)--(y); \draw(-0.2,1)--(0)--(-1,1); \draw(-0.25,1)--(0)--(-0.8,1);
\draw(-0.267,1)--(0)--(-0.75,1); \draw(-0.2727,1)--(0)--(-0.733,1);
\draw[red](-0.2764,1)--(0)--(-0.7236,1); \node at(-0.7236,1.15){$r_+$}; \node at(-0.2764,1.15){$r_-$};
\end{tikzpicture}\vspace{2mm}\\
$\mathcal{F}(B_{2,1})$ & $\mathcal{F}(B_{4,1})$ & $\mathcal{F}(B_{5,1})$
\end{tabular}
   \caption{Examples of the $g$-vector fans $\mathcal{F}(B_{b,c})$}
   \label{fig:rank 2 g-vector fans}
\end{figure}
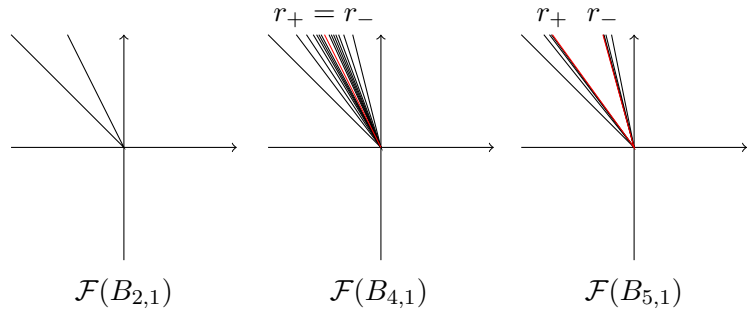
\begin{theorem}[{\cite[Theorem~10.6]{Re14}}]\label{thm:Reading}
If $B$ is of finite type, then $|\mathcal{F}(B)|=\mathbb{R}^n$.
\end{theorem}

We recall the following transition rule, which was conjectured in \cite[Conjecture~7.12]{FZ07}.

\begin{theorem}[{\cite[Corollary~5.5]{GHKK18}\cite[Proposition~4.2]{NZ12}}]\label{thm:linear transformation}
For $k \in \{1,\ldots,n\}$, $\mathcal{F}(\mu_k B)$ is obtained from $\mathcal{F}(B)$ by the map $(g_i)_{1 \le i \le n} \mapsto (g_i')_{1 \le i \le n}$, where
\[
g_i'=\begin{cases}
-g_k & \mbox{if} \ \ i=k,\\
g_i + [b_{ik}]_+g_k - b_{ik}\min(g_k,0) & \mbox{otherwise}.
\end{cases}
\]
\end{theorem}

This naturally provides the following result.

\begin{corollary}\label{cor:preserve lattice pt}
For $k\in\{1,\ldots,n\}$, $\mathbb{Z}^n\subset|\mathcal{F}(B)|$ if and only if $\mathbb{Z}^n\subset|\mathcal{F}(\mu_k(B))|$.
\end{corollary}

\begin{proof}
The assertion follows directly from Theorem~\ref{thm:linear transformation}.
\end{proof}

\section{Pull back of scattering diagrams}\label{sec:}

Let $B$ be an $n\times n$ skew-symmetrizable matrix. In this section, we prove the implication $(3)\Rightarrow(1)$ of Theorem~\ref{thm:main} by showing that if $B$ has a rank $2$ submatrix of infinite type, then there exists a lattice point outside the support $|\mathcal{F}(B)|$ (Proposition~\ref{prop:rank 2 no lattice pt}).

To prove Proposition~\ref{prop:rank 2 no lattice pt}, we use a pull-back of scattering diagrams. We refer to \cite{B17,GHKK18,Mu16} for the details of scattering diagrams. Roughly speaking, a \emph{scattering diagram} is a set of \emph{walls}, where a wall is a cone of codimension one in $\mathbb{R}^n$ together with some function. For the union of walls in a scattering diagram $\mathfrak{D}$, a connected component of its complement is called a \emph{chamber of $\mathfrak{D}$}. One can construct a scattering diagram $\mathfrak{D}(B)$ associated with $B$ and it relates to the $g$-fan $\mathcal{F}(B)$. We only state their properties which we need in this paper.

\begin{theorem}[{\cite[Theorem~0.8]{GHKK18}}]\label{thm:g-cone=chamber}
For a $g$-vector tuple $g$ for $B$, the interior of the $g$-cone $C(g)$ is a chamber of $\mathfrak{D}(B)$.
\end{theorem}

Theorem~\ref{thm:g-cone=chamber} means that a mutation of $g$-vector seeds corresponds an adjacent pair of chambers $C$ and $C'$ of $\mathfrak{D}(B)$. We say that $C'$ is the \emph{mutation of $C$ at the wall $\overline{C} \cap \overline{C'}$}. Thus any $g$-cone in $\mathcal{F}(B)$ is the closure of a chamber $C$ of $\mathfrak{D}(B)$ obtained from $C_0^{B}$ by a finite sequence of mutations, where $C_0^{B}$ is the interior of the cone spanned by $\mathbf{e}_1,\ldots,\mathbf{e}_n$,

For a subset $I \subset \{1,\ldots,n\}$, we consider a projection $\pi_I : \mathbb{R}^n \rightarrow \mathbb{R}^{|I|}$ given by $(r_i)_{1 \le i \le n} \mapsto (r_i)_{i \in I}$. We denote by $B_I$ the principal submatrix of $B$ indexed by $I$. The theorem below follows from the pull-back construction of scattering diagrams given in \cite[Theorem~33]{Mu16} for the skew-symmetric case. As observed in \cite{CL20}, this construction can be naturally extended to the skew-symmetrizable case.

\begin{theorem}[{\cite{CL20,Mu16}}]\label{thm:pull-back}
Let $\pi_I^{\ast}\mathfrak{D}(B_I)$ be a scattering diagram consisting of the walls $\pi_I^{-1}(W)$ for all walls $W$ of $\mathfrak{D}(B_I)$. Then each chamber of $\mathfrak{D}(B)$ is contained in some chamber of $\pi_I^{\ast}\mathfrak{D}(B_I)$.
\end{theorem}

\begin{proposition}\label{prop:rank 2 no lattice pt}
If there is a subset $I \subset \{1,\ldots,n\}$ such that $B_I=B_{b,c}$ with $bc \ge 4$, then
\[
\mathbb{Z}^n\setminus|\mathcal{F}(B)|\neq\emptyset.
\]
\end{proposition}

\begin{proof}
Let $C$ be the interior of a $g$-cone in $\mathcal{F}(B)$, which is a chamber of $\mathfrak{D}(B)$ by Theorem~\ref{thm:g-cone=chamber}. Let $C'$ be a cone obtained from $C$ by a single mutation. By Theorem~\ref{thm:pull-back}, the image $\pi_I(C)$ is contained in some chamber of $\mathfrak{D}(B_I)$, and $\pi_I(C')$ is contained in the same chamber or in one of its adjacent chambers.

Since $\pi_I(C_0^B)=C_0^{B_I}$ and there are infinitely many chambers of $\mathfrak{D}(B_I)$ converging to $r_{\pm}$ as in Example~\ref{exam:rank 2 g-fan}, the support $|\mathcal{F}(B)|$ does not intersect the cone spanned by $\pi_I^{-1}(r_{\pm})$. Moreover, this cone contains a lattice point by Example~\ref{exam:rank 2 g-fan}. Therefore, $\mathbb{Z}^n\setminus|\mathcal{F}(B)|\neq\emptyset$.
\end{proof}

We are ready to prove Theorem~\ref{thm:main}.

\begin{proof}[Proof of Theorem~\ref{thm:main}]
The implication $(1)\Rightarrow(2)$ follows from Theorem~\ref{thm:Reading}, and $(2)\Rightarrow(3)$ clearly holds. Finally, we prove $(3)\Rightarrow(1)$. By Theorem~\ref{thm:fin 2-fin}, if $B$ is not of finite type, it is mutation equivalent to $B'$ such that $B'_I=B_{b,c}$ with $bc\ge 4$ for some subset $I \subset \{1,\ldots,n\}$. By Proposition~\ref{prop:rank 2 no lattice pt}, $\mathbb{Z}^n\setminus|\mathcal{F}(B')|\neq\emptyset$, which implies that $\mathbb{Z}^n\setminus|\mathcal{F}(B)|\neq\emptyset$ by Corollary~\ref{cor:preserve lattice pt}. This completes the proof.
\end{proof}

\medskip\noindent{\bf Acknowledgements}.
The author would like to thank Mohamad Haerizadeh for helpful discussions. He was supported by JSPS KAKENHI Grant Numbers JP21K13761.

\bibliographystyle{alpha}
\bibliography{bib}

\end{document}